\documentclass[12pt]{article}

\usepackage{amssymb}

\usepackage{amsfonts}

\usepackage{amsmath}

\usepackage{amsthm}

\usepackage{graphicx}
\def\bd{\mathrm{bd}}
\def\inn{\mathrm{int}}
\def\bs{{{\bigskip}}}
\def\cb{{{convex body}}}
\def\cbs{{{convex bodies}}}
\def\R{{{\rm I\! R}^d}}

\def\Rp{{{\rm I\! R}^2}}
\def\Rs{{\mathbb{R} ^3}}
\def\o{{{\bf 0}}}
\def\conv{{\rm conv}}

\newtheorem{thm}{Theorem}[section]

\newtheorem{lm}[thm]{Lemma}

\newtheorem{co}[thm]{Corollary}

\newtheorem{rmk}[thm]{Remark}

\begin{document}

\title{\bf Locating diametral points}

\author{Jin-ichi Itoh \and Costin V\^{\i}lcu \and Liping Yuan \and Tudor Zamfirescu}


\maketitle

\begin{abstract}

Let $K$ be a convex body in $\mathbb{R} ^d$, with $d = 2,3$. We
determine sharp sufficient conditions for a set $E$ composed of $1$,
$2$, or $3$ points of ${\rm bd}K$, to contain at least one endpoint
of a  diameter of $K$ (for $d=2,3$).
We extend this also to convex surfaces, with their intrinsic metric.
Our conditions are upper bounds on the sum of the complete angles
 at the points in $E$.  We also show that
such criteria do not exist for $n\geq 4$ points.

\smallskip

AMS Math. Subj. Classification (2010): 52{\sc A}10, 52{\sc A}15,
53{\sc C}45.

\smallskip


\end{abstract}


\section {Introduction}

\medskip

The tangent cone at a point $x$ in the boundary bd$K$ of a convex
body $K$ can be defined using only neighborhoods of $x$ in bd$K$.
So, one doesn't normally expect to get global information about $K$
from the size of the tangent cones at one, two or three points.
Nevertheless, in some cases this is what happens!

\medskip

A {\it convex body} $K$ in $\mathbb{R} ^d$ is a compact convex set
with interior points in $\mathbb{R} ^d$; we shall consider only the
cases $d = 2,3$. A {\it convex surface} in $\mathbb{R} ^3$ is the
boundary of a convex body in $\mathbb{R} ^3$.

Let $S$ be a convex surface  and $x$ a point  in $S$. Consider
homothetic dilations of $S$ with the centre at $x$ and coefficients
of homothety tending to infinity. The limit surface is
called the {\it  tangent cone at} $x$ (see \cite{al}), and is
denoted by $T_x$.

If $K$ is a planar convex body then a tangent cone is an angle, and
its measure is the angle-measure.

If $K$ is a convex body in $\mathbb{R} ^3$ then the tangent cone at
$x \in {\rm bd}K$
can be unfolded in the plane, producing an angle the measure of
which is {\it the complete angle} at $x$, denoted by $\theta_x$.

\medskip


Denote by $\rho$ the intrinsic metric on the convex surface $S$
(which is derived from the ambient Euclidean distance).

We shall  call {\it diameter} each line-segment in $K$, or arc in
$S$, of length equal to the
 extrinsic, respectively intrinsic, maximal distance
between pairs of points in $K$ or in $S$.


An endpoint of some (intrinsic or extrinsic) diameter is called an
(intrinsic, respectively extrinsic) {\it diametral point}.

In this paper, we provide criteria for finding extrinsic diametral
points in convex bodies $K \subset \mathbb{R} ^d$, $d=2,3$, and
criteria for finding intrinsic diametral points in convex surfaces
$S={\rm bd}K \subset \mathbb{R} ^3$. Our criteria consist of upper
bounds on the sum of the
complete angles  at $1$, $2$, or $3$ points.

We also show that such criteria do not exist for $n\geq 4$ points.

\medskip

Related to our results in Section \ref{sf3} is the following one,
obtained by J. Itoh and C. V\^\i lcu \cite{iv-mn}. Each point $y$ in
a convex surface $S$ with complete angle $\theta _y \leq \pi$ is a
farthest point on $S$, i.e., $y$ is at maximal intrinsic distance
from some point  in $S$.

\medskip

Passing from planar convex bodies to convex surfaces is not always
obvious. For example, while the  diameter of a convex polygon $P$
(in the plane, diameter means extrinsic diameter) with $n$ vertices
can be computed in time $O(n)$ \cite{sh}, the intrinsic diameter of
a convex polyhedral surface in $\Rs$ with $n$ vertices can be
computed in time $O(n^8 \log n)$ \cite{agr}.

Also, it is well-known that  diameters of  convex polygons must join
vertices, but this is not always true for  geodesic diameters of
convex polyhedral surfaces \cite{osc}.





\medskip

There is a nice connection between the lengths of extrinsic and
intrinsic diameters of a convex surface, considered by several
authors, see \cite{ma}, \cite{mak}, \cite{zal}: for any convex
surface $S$, the former is not larger than $\pi/2$ times the latter,
and equals it
if and only if $S$ is a  surface of revolution of constant
width.

Our results provide another connection.
The endpoints of extrinsic and intrinsic diameters of convex bodies
and surfaces are in general distinct; yet, in some cases, they can
be found in the same set, see the remarks at the end of the paper.

\bs
 A pair of points {\it sees a line-segment under the angle} $\alpha$
if the sum of the two angles under which they see the line-segment
equals $\alpha$.

Let $\sigma$ be an extrinsic diameter of the convex body $K$.  A
pair of points $u, v \in K \setminus \sigma$ is said to be a
$\sigma${\it -separated pair} if the line-segment $uv$  meets
$\sigma$ \cite{z-diam}.

A {\it  segment} on the convex surface $S$ is an arc (path) on $S$
realizing the intrinsic distance between its endpoints. If $\sigma$
is an intrinsic diameter, i.e. a longest segment, of  $S$, then a
pair of points $u, v \in S \setminus \sigma$ is said to be
$\sigma${\it -separated} if some segment from $u$ to $v$ meets
$\sigma$ \cite{z-diam}.

For $M\subset\R$, we denote by $\overline{M}$ its affine hull, by
$\inn M$ the relative interior of $M$ (i.e., in the topology of
$\overline{M}$) and by $\bd M$ the relative boundary of $M$.

For distinct $x,y\in \R$, let $xy$ be the line-segment from $x$ to
$y$; thus, $\overline{xy}$ is the line through $x, y$. We put
$x_1...x_n = \conv \{x_1,..., x_n\}.$


\section {Planar convex bodies}

\label{planar}

Let $K$ be a planar convex body and $x$ a boundary point of $K$.

We denote  by $X$ the angle of bd$K$ at $x$ towards $K$ (so $X \leq
\pi$), and keep this habit for any boundary point; so, $Y$ is the
angle at $y$, and so on.

\medskip





We shall repeatedly use the next result.

\begin{lm}[T. Zamfirescu \cite{z-diam}]

\label{view}

For any diameter $uv$ of a planar convex body, every $uv$-separated
pair sees $uv$ under an angle not less than $5\pi/6$.

\end{lm}









\begin{lm}

\label{4}

Assume in the convex quadrilateral $Q=xyzw$ we have $X + Y \leq
\pi$. Then at least one of the vertices $x,y$ is a diametral point
of $Q$.

\end{lm}

\begin{proof}

Assume $x,y$ are not diametral points of $Q$. Then the side $zw$ is
 longer than the diagonals $xz$ and $yw$, whence
$W < \angle wxz<X$ and $Z < \angle wyz<Y$. It follows that

$$2 \pi = X+Y+Z+W < 2 (X+Y) \leq 2 \pi,$$
absurd.

\end{proof}

\begin{thm}

\label{23}

Let $K$ be a planar convex body.

$(i)$ Any point $x \in\bd K$ with $X \leq \pi/3$ is a diametral
point of $K$. If $K$ has two such points, they determine a diameter
of $K$.

$(ii)$ Among any two points $x,y \in\bd K$ with $X+Y \leq 5\pi/6$
there exists a diametral point of $K$.

$(iii)$ Among any three points $x,y,z \in\bd K$ with $X+Y + Z \leq
4\pi/3$ there exists a diametral point of $K$.

\end{thm}

\begin{proof}

$(i)$ Assume the existence of a diameter $yz$ of $K$, with $y, z$
different from $x$. It follows that in the triangle $xyz$ the angle
at $x$ is not smaller than the other two, whence the triangle is
equilateral. Hence, $xy$ and $xz$ are diameters, too. Thus, (i) is
proven.


\medskip

For the rest of the proof (parts $(ii)$ and $(iii)$), assume the
conclusion does not hold, and let $uv$ be a diameter of $K$.

\medskip

$(ii)$ If $x$ and $y$ are not $uv$-separated, then we have the
quadrilateral $xyvu$. By Lemma \ref{4}, $X+Y \ge\angle uxy + \angle
xyv> \pi$, which contradicts our hypothesis.

So $x$ and $y$ are $uv$-separated; then, by Lemma \ref{view}, $X+Y
\geq 5\pi/6$. This and the hypothesis imply $X+Y = 5\pi/6$.

Assume that $K$ is not the quadrilateral $xuyv$. Then the
sum of the angles of $xuyv$ at $x$ and $y$ is less than $5\pi/6$, in
contradiction with Lemma \ref{view}, applied to $xuyv$.

So $K$ is the quadrilateral $xuyv$. Slightly moving $x$ out of $K$
along the line  $\overline{xy}$ would provide quadrilaterals
$K'=x'uyv$ with $x',y$ $uv$-separated and $X'+Y < 5\pi/6$. By Lemma
\ref{view}, $uv$ is no longer a diameter of $K'$, so  $x'$  is a
diametral point of $K'$. Now, let $x'$ converge back towards $x$.
Then  $K'\to K$, which implies that  $x$ is a diametral point of
$K$, contradicting our assumption.

\medskip

$(iii)$ Assume first that $x, y, z$ are all on one side of
$\overline{uv}$. Lemma \ref{4} gives

$$X + Y > \pi, \hspace{0.5cm} X+ Z > \pi, \hspace{0.5cm} Y+ Z > \pi,$$
so $X+Y+Z > 3 \pi/2$, contradicting $X+Y + Z \le 4\pi/3$.

Hence,  we can assume that $x, y$ are on one side of $\overline{uv}$
and $z$  on the other side. The previous case $(ii)$ and Lemma
\ref{4} imply

$$X + Y > \pi, \hspace{0.5cm} X+ Z > 5\pi/6, \hspace{0.5cm} Y+ Z > 5\pi/6.$$

Summing up, we get $X+Y+Z > 4 \pi/3$, which contradicts the
hypothesis.


\end{proof}



All bounds in Theorem \ref{23} are sharp, as one can see from the
following examples.

$(i)$ Consider an isosceles triangle $\Delta=xyz$ with equal sides
$\|x-y\|=\|x-z\|$ and $X = \pi/3 + \varepsilon$, with $\varepsilon$
arbitrarily small. Clearly, $x$ is not a diametral point of
$\Delta$.

$(ii)$ Consider a convex quadrilateral $Q'=x'uy'v$ with $\| x'-v \|
= \| x'-y' \| = \| y'-v \| = \| u-v \|$ and $\|u-x'\|=\|u-y'\|$.
Then, in $Q'$, $X'=Y'=U/2=5\pi/12$.

Let $x$ and $y$ be interior points of $Q'$, on the line
$\overline{x'y'}$, arbitrarily close to $x$ and $y$, respectively.
Then, in $Q=xuyv$, we have $X+Y=5\pi/6 + \varepsilon$, with
$\varepsilon$ arbitrarily small; moreover,  $uv$ is the unique
diameter of $Q$.

$(iii)$ Consider an equilateral triangle $\Delta=uvz'$ and let $m$
be the midpoint of $uv$. Take points $x,y$ on the circle of diameter
$uv$, separated from $z'$ by $\overline{uv}$, such that $xy\|uv$ and
$x$ is arbitrarily close to $u$. Then, in $xuz'y$, $X=Y=\pi/2 +
\varepsilon$, with $\varepsilon$ arbitrarily small.

Take a point $z$ on $z'm$, such that  $\|z-z'\|$ equals the distance
between the parallel lines $\overline{xy}$ and $\overline{uv}$. Of
course, in $xuzvy$, $Z = \angle uzv = \pi/3 + \varepsilon'$, with
$\varepsilon'>0$.

Then $X+Y + Z= 4\pi/3 + 2\varepsilon + \varepsilon'$, and
$2\varepsilon + \varepsilon'$ converges to $0$ as $x \to x'$.
Moreover,  $uv$ is the unique diameter of $xyvzu$.

\begin{co}

\label{cor1}

If the planar convex body $K$, symmetric about $\o$, has a boundary
point $x$ with $X \leq 5\pi/12$, then $x(-x)$ is a diameter of $K.$


\end{co}

\begin{proof}

The sum of the angles at $x$ and $-x$ is not larger than $5\pi/6$.
Now,
by Theorem \ref{23} (ii),  $x$ or $x'$ is a diametral point of $K$.
But, by Theorem 4 in \cite{z-diam}, the endpoints of each diameter
of $K$ are symmetric with respect to $\o$. So, $x(-x)$ is a diameter
of $K$.

\end{proof}

The above approach cannot be extended to $n \geq 4$ points.

\begin{rmk}

\label{no4}

There is no non-trivial constant $d(n)$ depending only on $n\ge 4$,
to guarantee that, for any planar convex body $K$, among any $n$
points $x_1, ... , x_n$ in ${\rm bd}K$ with $\sum_{i=1}^n X_i \leq
d(n)$ there exists a diametral point of $K$.

\end{rmk}

\begin{proof}
Suppose that such a constant $d(n)$ does exist.

Notice that any points $x_1, ... , x_n$ in the boundary of any
planar convex body $K$, form a convex $n$-gon; hence $\sum_{i=1}^n
X_i \geq (n-2)\pi$, and therefore $d(n) \geq (n-2)\pi$.

Next we show that, for any $\varepsilon >0$, there exist a planar
polygon $P$ and $n$ vertices of $P$ with $\sum_{i=1}^n X_i <
(n-2)\pi + \varepsilon$, none of which is a diametral point of $P$.
This implies $d(n) \leq (n-2)\pi$, hence necessarily $d(n) =
(n-2)\pi$. In this case, $K$ is precisely the convex $n$-gon with
vertices $x_1, ... , x_n$ and, trivially, at least two of them are
diametral points.


Let $uv$ be a diameter of a  circle ${C}$. Consider $n\geq 4$ points
$x_1, ... , x_n$ on ${C}$, at least two of them on each side of the
line $\overline{uv}$, such that no two of them are diametrally
opposite. Let $x_1, ... , x_k$ be on one side and $x_{k+1}, ... ,
x_n$ on the other side of $\overline{uv}$.

Of course, in the $n$-gon $x_1...x_n$,  we have $\sum_{i=1}^n \angle
x_i x_{i+1} x_{i+2} = (n-2)\pi$, where indices $i$ are taken modulo
$n$. Let $P$ denote the $(n+2)$-gon  $uvx_1...x_n$. By taking
$x_1,x_n$ close to $u$, and $x_k, x_{k+1}$ close to $v$, we get
$\sum_{i=1}^n X_i < (n-2)\pi + \varepsilon$, with $\varepsilon$
arbitrarily small. But $P$ has the unique diameter $uv$.
\end{proof}


\section {Convex bodies in $\mathbb{R} ^3$}

\label{bd3}

We obtain here
  results similar to Theorem 2.3, locating diametral points of \cbs\ in $\Rs$.

We denote by $\theta_x$ the complete angle at the point $x$ in
bd$K$, and by $\omega_x$ the curvature at $x$; hence, $\omega_x =
2\pi - \theta_x$.

\begin{thm}

\label{3n}

Let $K$ be a convex body in $\mathbb{R} ^3$.

$(i)$ Any point $x \in {\rm bd} K$ with $\theta_x \leq 2\pi/3$ is a
diametral point of $K$. If $K$ has two such points, they determine a
diameter of $K$.

$(ii)$ Among any two points $x,y \in {\rm bd} K$ with $\theta_x +
\theta_y \leq 3\pi/2$ there  exists a diametral point of $K$.

$(iii)$ Among any three points $x,y,z \in {\rm bd} K$ with $\theta_x
+ \theta_y + \theta_z \leq 9\pi/4$ there exists a diametral point of
$K$.

\end{thm}

\begin{proof}


$(i)$ Assume there exists $K \subset \mathbb{R} ^3$ and a point
$x\in\bd K$ with $\theta_x \le 2\pi/3$, which is not diametral.

Let $uv$ be a diameter of $K$. In the planar \cb\
$K\cap\overline{xuv}$, the angle $X$ at $x$ must be at most
$\theta_x/2\le \pi/3.$ By Theorem 2.4 $(i)$, $X>\pi/3,$ and a
contradiction is obtained.

\bs
 $(ii)$ Assume there exists $K \subset \mathbb{R} ^3$ and points
$x,y$ on bd$K$ with $\theta_x + \theta_y \leq 3\pi/2$, none of which
is a diametral point of $K$.

Let $uv$ be a diameter of $K$. We consider the (possibly degenerate)
tetrahedron $T=uvxy$.

We unfold $xuv\cup yuv$ on a plane, with $x, y$ coming on different
sides of $\overline{uv}$. The resulting quadrilateral $Q$ has angles
$X$, $Y$, $U$, $V$ at the points corresponding to $x$, $y$, $u$,
$v$, respectively.
 Now, unfold $uxy\cup vxy$ on a plane, with $u, v$ coming on
different sides of $\overline{xy}$. The resulting quadrilateral $Q'$
has angles $X'$, $Y'$, $U'$, $V'$ at the points corresponding to
$x$, $y$, $u$, $v$. In $Q'$, the length of the diagonal
corresponding to $uv$ equals at least $\|u-v\| > \|x-y\|.$ By
Theorem 2.4 $(ii)$, $X'+Y'>5\pi/6,$ whence $U'+V'< 2\pi -
(5\pi/6)=7\pi/6.$

We have

$$X+X'+Y+Y'+U+U'+V+V'=4\pi$$
in $\bd T$.

Since $X+X'+Y+Y'\le \theta_x+\theta_y \le 3\pi/2,$ we have
$U+U'+V+V'\ge 5\pi/2.$ This, together with the inequality
$U'+V'<7\pi/6$ obtained above, yields $U+V> 4\pi/3.$ This implies
$X+Y< 2\pi/3$. Hence, $X<\pi/3$ or $Y<\pi/3$. Thus, $uv$ cannot be a
longest side,  in $xuv$ or in $yuv$, and a contradiction is
obtained.

\bs
 $(iii)$
Suppose
 $\theta_x
+ \theta_y + \theta_z \leq 9\pi/4$, but there is no diametral point
among $x, y, z$. Then, by $(ii)$, $\theta_x + \theta_y > 3\pi/2$,
$\theta_y + \theta_z >3\pi/2$, $\theta_z + \theta_x > 3\pi/2$. It
follows that $2\theta_x + 2\theta_y + 2\theta_z > 9\pi/2$, in
contradiction with our hypothesis.




%



%


%


\end{proof}





\begin{thm}

\label{cor32}

If the  convex body $K$, symmetric with respect to $\o$, has a
boundary point $x$ with $\theta_x \leq 5\pi/6$, then $x(-x)$ is a
diameter of $K$.


\end{thm}

\begin{proof}
Obviously, $\theta_x=\theta_{-x}$. Assume that $x(-x)$ is not a
diameter. Then consider a diameter, which, by Theorem 4 in [11],
must join diametrally opposite points. Let $y(-y)$ be that diameter.
In the parallelogram $xy(-x)(-y)$, the diagonal $y(-y)$ is the
unique diameter of it, so $x$ and $-x$ are not diametral points. By
Theorem 2.3 $(ii)$, $\angle yx(-y) +\angle y(-x)(-y)> 5\pi/6$.

But $\angle yx(-y) +\angle y(-x)(-y)\le (\theta_x/2)
+(\theta_{-x}/2)\le 5\pi/6$, and we got a contradiction.

\end{proof}







\section{Convex surfaces}

\label{sf3}

In this section we investigate intrinsic diameters on convex
surfaces. We obtain results similar to those in Sections 2 and 3.
Roughly speaking, as soon as the curvature concentrated at some
points is large enough, they become eligible as diametral points.


\begin{lm}[The Pizzetti-Alexandrov comparison theorem (\cite{al}, p. 132)]

\label{comp}

The angles of any geodesic triangle in a convex surface are not
smaller than the corresponding angles of the Euclidean triangle with
the same side-lengths.

\end{lm}


Lemma 4.2 below follows from Alexandrov's {\it
Konvexit\"{a}tsbedingung} (\cite{al}, p. 130).

\begin{lm}[]

\label{view-s}

Consider a convex surface $S$. Let $abc\subset S$ and
$a'b'c'\subset\Rp$ be two triangles as in Lemma $4.1$. If $d\in bc$,
$d'\in b'c'$ and $\rho(b, d)=\|b'-d'\|$, then $\rho(a, d)\ge
\|a'-d'\|.$

\end{lm}

The following statement is well-known. For a thorough introduction
to the theory of critical points for distance functions, see
\cite{gro}.

\begin{lm}


Each endpoint of a diameter on a convex surface is critical with
respect to the other.
Consequently, each digon determined by two  diameters, with no third
diameter passing through its interior, has both angles at most
$\pi$.

\end{lm}

\begin{thm}

\label{3d}

Let $S$ be a convex surface.

$(i)$ Any point $x \in S$ with $\theta_x \leq 2\pi/3$ is a diametral
point of $S$. If $S$ has two such points, they determine a diameter
of $S$.

$(ii)$ Among any two points $x,y \in S$ with $\theta_x + \theta_y
\leq 5\pi/3$ there exists a  diametral point of $S$.

$(iii)$ Among any three points $x,y,z \in S$ with $\theta_x +
\theta_y + \theta_z \leq 5\pi/2$ there exists a  diametral point of
$S$.

\end{thm}

\begin{proof}

$(i)$ Assume a point $x$ on $S$ verifies $\theta_x\le 2\pi/3$ and is
not a diametral point of $S$. Let $yz$ be a diameter of $S$. There
are two geodesic triangles with vertices at $x, y, z$ on $S$, at
least one of which has an angle not larger than $\pi/3$ at $x$. A
contradiction now follows from
 the assumption that $x$ is not a diametral point,
Lemma \ref{comp} and Theorem \ref{23} $(i)$.

Assume now that there are $x,y \in S$ with $\theta_x, \theta_y \leq
2\pi/3$, and take $z \in S \setminus \{x,y\}$. Join $x$, $y$ and $z$
by  segments to form two triangles on $S$. At least one of them has
its angle at $x$ not larger than $\pi/3$, so $yz$ is not a diameter
of $S$ or $xy$ is a diameter, by the preceding argument.
Analogously, $xz$ is not a diameter of $S$ or $xy$ is a diameter.

Since this holds for any $z\in S$ and $x,y$ are diametral points of
$S$, $xy$ must be a diameter of $S$.

\medskip

For the rest of the proof, assume the conclusions are false and let
$uv$ be a diameter of $S$.

The segments joining $u$ and $v$ determine on $S$ one or several
digons.

\medskip

$(ii)$ The points $x,y$ are not inside one digon, say $D$,
determined by segments from $u$ to $v$. Indeed, by Lemma 4.3, the
total curvature of the interior of $D$ is at most $2 \pi$,
hence
$$ 2\pi \geq \omega_x + \omega_y = 4\pi - (\theta_x + \theta_y ) \geq \frac73 \pi> 2\pi,$$
absurd.

Therefore, the points $x,y$ are in distinct digons, and so $x$ and
$y$ are $uv$-separated, for some diameter $uv$. Let $\{w\}=uv\cap
xy.$ Consider the points $x',y',u',v',w'$ in $\Rp$ such that $w'\in
u'v'$, $x'y'\cap u'v'\neq\emptyset$,
$\|u'-v'\|=\rho(u, v),$ $\|u'-x'\|=\rho(u, x),$ $\|u'-y'\|=\rho(u,
y),$ $\|v'-x'\|=\rho(v,x),$ $\|v'-y'\|=\rho(v,y),$
$\|u'-w'\|=\rho(u,w).$ By Lemma 4.2, $\|x'-w'\|\le \rho(x, w)$ and
$\|y'-w'\|\le \rho(y, w)$.

By Lemma 4.1, $\angle uxv \ge \angle u'x'v'$ and $\angle uyv \ge
\angle u'y'v'$.

But
$$\|x'-y'\|\le \|x'-w'\|+\|w'-y'\|\le
\rho(x,w)+\rho(w,y)=\rho(x,y)\le\rho(u,v)=\|u'-v'\|.$$
By Theorem 2.3 $(ii)$,
$$\angle uxv + \angle
uyv \ge \angle u'x'v' + \angle u'y'v' > 5\pi/6.$$
 But
 $$\angle uxv +
\angle uyv \le(\theta_x/2)+(\theta_y/2)\le 5\pi/6,$$
and a contradiction is obtained.








%



\medskip

$(iii)$ Notice that the points $x,y,z$ cannot be all in the same
digon determined by segments from $u$ to $v$. Indeed, for three
points $x,y,z$ in the same digon, we have, by Lemma 4.3, $\omega_{x}
+ \omega_{y} + \omega_{z} \leq 2\pi$, hence $\theta_{x} + \theta_{y}
+ \theta_{z} \geq 4\pi$,  contradicting the hypothesis.

Assume first that $x,y$ are in one digon, and $z$  in another one.
Then, by $(ii)$, $\theta_x + \theta_z > 5\pi/3$ and $\theta_y +
\theta_z > 5\pi/3$.
 At $(ii)$ we saw that
$\theta_x + \theta_y \geq 2\pi$.
Summing up these inequalities,  we get $\theta_x + \theta_y +
\theta_z > 8\pi/3 > 5\pi/2$, impossible.

Hence, $x,y,z$ are  in  different digons.
By $(ii)$, we have $\theta_x + \theta_y > 5\pi/3$, $\theta_y +
\theta_z > 5\pi/3$, and $\theta_x + \theta_z > 5\pi/3$, hence
$\theta_{x} + \theta_{y} + \theta_{z} > 5 \pi/2$, in contradiction
with the hypothesis.

\end{proof}












\begin{co}

\label{cor3}

If the  convex surface $S$, symmetric about $\o$, has a point $x$
with $\theta_x \leq 5\pi/6,$ then there exists a  diameter of $S$
from $x$ to $-x$.



\end{co}

\begin{proof}
Since $\theta_x=\theta_{-x}$, we have $\theta_x+\theta_{-x}\le
5\pi/3.$ Now, Theorem \ref{3d} $(ii)$ implies that  $x$ or $x'$ is a
diametral point of $S$. By Proposition 6 in \cite{v00}, each
diameter of $S$ is realized between diametrally opposite points.

\end{proof}









The endpoints of extrinsic and intrinsic diameters of convex bodies
or surfaces are in general distinct.

The hypotheses of Corollaries \ref{cor32} and \ref{cor3}, are the
same. Also,  those of Theorem \ref{3d} and Theorem \ref{3n} might be
simultaneously verified. In these cases, endpoints of both extrinsic
and intrinsic diameters of $S={\rm bd}K$ can be found in the same
subset of $S$ composed by $1$, $2$, or $3$ points.


\bs
 {\bf Conjecture.} In Theorem 3.1 $(ii)$, the inequality
 $\theta_x+\theta_y\le 5\pi/3$ suffices to guarantee
the existence of a diametral point in $\{x, y\}$.

\bs
 {\bf Open questions.} Are the bounds $9\pi/4$ in Theorem 3.1
 $(iii)$ and $5\pi/2$ in Theorem 4.4 $(iii)$ optimal?

\section*{Acknowledgements.} The first author  was partially supported by a Grant-in-Aid for Scientific
Research (C) (No. 17K05222), Japan Society for Promotion of Science.
 The last two authors gratefully acknowledge financial support by NSF of China
(11871192, 11471095). The last three authors direct their thanks to
the Program for Foreign Experts of Hebei Province (No. 2019YX002A).
The research of the last author was also partly supported by the
International Network GDRI ECO-Math.



\medskip

Jin-ichi Itoh

\noindent{\footnotesize School of Education, Sugiyama Jogakuen
University

\noindent 17-3 Hoshigaoka-motomachi, Chikusa-ku, Nagoya, 464-8662
Japan}

{\small \hfill j-itoh@sugiyama-u.ac.jp}

\medskip

Costin V\^\i lcu

\noindent{\footnotesize {\sl Simion Stoilow} Institute of
Mathematics of the Roumanian Academy

\noindent P.O. Box 1-764, 014700 Bucharest, Roumania}

{\small \hfill Costin.Vilcu@imar.ro}

\medskip

Liping Yuan

\noindent{\footnotesize College of Mathematics and Information
Science,

\noindent Hebei Key Laboratory of Computational Mathematics and
Applications,

\noindent Hebei Normal University,

\noindent 050024 Shijiazhuang, P.R. China.}

{\small \hfill lpyuan@hebtu.edu.cn}

\medskip

Tudor Zamfirescu

\noindent{\footnotesize Fachbereich Mathematik, Universit\"at
Dortmund,

\noindent{44221 Dortmund, Germany, {and}

\noindent Roumanian Academy, Bucharest, {and}

\noindent College of Mathematics and Information Science,

\noindent Hebei Normal University,

\noindent 050024 Shijiazhuang, P.R. China.}

{\small \hfill tudor.zamfirescu@mathematik.tu-dortmund.de}


\begin{thebibliography}{99}


\bibitem{al} A.D. Alexandrov, {\it Die innere Geometrie der konvexen Fl\"achen}, Akademie-Verlag, Berlin, 1955


\bibitem{agr} P. K. Agarwal, B. Aronov, J. O'Rourke and C. A. Schevon,
{Star unfolding of a polytope with applications}, {\it SIAM J.
Comput.} {\bf 26} (1997), 1689-1713


\bibitem{gro} K. Grove, {Critical point theory for distance functions},
{\it Amer. Math. Soc. Proc. of Symposia in Pure Mathematics}, {\bf
54} (1993), 357-385


\bibitem{iv-mn} J. Itoh and C. V\^\i lcu, {Criteria for farthest points on convex surfaces},
{\it Math. Nachr.} {\bf 282} (2009), 1537-1547


\bibitem{mak} E. Makai, Jr., {On the geodesic diameter of convex surfaces}, {\it Period. Math. Hungar.} {\bf 4}
 (1972), 157-161


\bibitem{ma} N. P. Makuha, {A connection between the inner and
the outer diameters of a general closed convex surface} (in
Russian), {\it Ukrain. Geometr. Sb. Vyp.} {\bf 2} (1966), 49-51


\bibitem{osc} J. O'Rourke and C. A. Schevon, {Computing the geodesic diameter of a 3-polytope},
{\it Proc. 5th ACM Symp. Comput. Geom.} (1989), 370-379


\bibitem{sh} M.I. Shamos, {\it Computational geometry}, Ph.D. thesis, Yale University, 1978


\bibitem{v00} C. V\^\i lcu, {On two conjectures of Steinhaus}, {\it Geom. Dedicata} {\bf 79} (2000), 267-275


\bibitem{zal} V. A. Zalgaller, {The geodesic diameter of a body of constant width}, {\it J. Math. Sci. (N.Y.)}
 {\bf 161} (2009), 373-374


\bibitem{z-diam} T. Zamfirescu, {Viewing and realizing diameters}, {\it J. Geom.} {\bf 88} (2008), 194-199


\end{thebibliography}
\end{document}